\documentclass[12pt,twoside]{amsart} 
\usepackage{amssymb,color,tikz}
\usetikzlibrary{matrix} 

\usepackage{enumitem}

\usepackage{xcolor}

\usepackage{cleveref}    

\usepackage{comment}

\numberwithin{equation}{section} 

\newcommand {\PP}{\mathbb{P}}
\newcommand{\ffi}{\varphi}
\DeclareMathOperator{\Ann}{Ann} 
\DeclareMathOperator{\codim}{codim} 
\DeclareMathOperator{\coker}{coker}
\DeclareMathOperator{\Jac}{Jac}
\DeclareMathOperator{\Proj}{Proj}

\def\cocoa{{\hbox{\rm C\kern-.13em
      o\kern-.07em C\kern-.13em o\kern-.15em A}}}

\newtheorem{theorem}{Theorem}[section]
\newtheorem{lemma}[theorem]{Lemma}
\newtheorem{proposition}[theorem]{Proposition}
\newtheorem{corollary}[theorem]{Corollary}
 
\newtheorem{problem}[theorem]{Problem} 
\theoremstyle{definition}
\newtheorem{definition}[theorem]{Definition} 
\newtheorem{remark}[theorem]{Remark}
\newtheorem{example}[theorem]{Example}
\newtheorem{notation}[theorem]{Notation}
\definecolor{MyDarkGreen}{cmyk}{0.7,0,1,0}

\begin{document}

\title[Jacobian schemes arising from hypersurface arrangements in $\mathbb P^n$]%
{Jacobian schemes arising from hypersurface arrangements in $\mathbb P^n$}

\author[J.\ Migliore]{Juan Migliore${}^*$}
\address{ Department of Mathematics, University of Notre Dame, Notre
  Dame, IN 46556, USA} \email{Juan.C.Migliore.1@nd.edu} \author[R.\
\author[U.\ Nagel]{Uwe Nagel${}^{**}$} \address{Department of
  Mathematics, University of Kentucky, 715 Patterson Office Tower,
  Lexington, KY 40506-0027, USA} \email{uwe.nagel@uky.edu}
\thanks{Printed \today\\
  Migliore was partially supported by grants from the Simons Foundation
  \#309556 and \#839618.\\
  Nagel was partially supported by Simons Foundation grant \#636513.  \\
  The authors are very grateful to the referee for careful reading of the paper, and in particular for asking a question that led to the discovery of an oversight and a strengthening of the results of this paper.
} \pagestyle{plain}
\begin{abstract}
Freeness is an important property of a hypersurface arrangement, although its presence is not well understood. A hypersurface arrangement in $\PP^n$ is free if $S/J$ is Cohen-Macaulay (CM), where $S = K[x_0,\ldots,x_n]$ and $J$ is the Jacobian ideal. We study three related unmixed ideals: $J^{top}$, the intersection of height two primary components,  $\sqrt{J^{top}}$, the radical  of $J^{top}$, and when the $f_i$ are smooth we also study $\sqrt{J}$. Under mild hypotheses, we show that these ideals are CM. This establishes a full generalization of an earlier result with Schenck from  hyperplane arrangements to hypersurface arrangements. If the hypotheses fail for an arrangement in projective $3$-space, the Hartshorne-Rao module measures the failure of CMness. We establish consequences for the even liaison classes of $J^{top}$ and $\sqrt{J}$.

\end{abstract}
\maketitle
\section{Introduction}

Hypersurface arrangements play a role in several topics and have been intensely investigated using tools from algebra, algebraic geometry, combinatorics, topology and others. In this paper, we demonstrate that methods from liaison theory can be employed to elucidate properties of hypersurface arrangements.

A hypersurface arrangement $\mathcal A$ in projective space is a union  of distinct hypersurfaces $F_1,\ldots, F_s$ in $\mathbb P^n = \PP^n_k $, where $k$ denotes a field of characteristic zero. Note that $\mathcal A$ is defined by a product $f = \prod_{i=1}^s f_i$ of  forms $f_i$ in $S = K[x_0,\dots,x_n]$ defining $F_i$.  Let $J = \Jac(f) =  \langle \frac{\partial f}{\partial x_0}, \dots, \frac{\partial f}{\partial x_n} \rangle$ be the Jacobian ideal of $f$. 
We say that $\mathcal A$ is {\it free} if $J$ is a saturated ideal defining an arithmetically Cohen-Macaulay (ACM) codimension two subscheme $X$ of $\mathbb P^n$ (see Definition \ref{def of ACM}). It is an important question in hypersurface arrangement theory to ask what conditions force an arrangement of hypersurfaces to be free. The case that is most studied, of course, is the case of hyperplanes. 

Failure to be free can have several causes. Among them: (a)  the  ideal $J$ may fail to be saturated, (b)  $J$ may be saturated, but fails to be unmixed (this necessarily implies the existence of embedded components of $X$), or (c) it may even happen that $J$ (or its saturation)  is unmixed but $X$ fails to be ACM  for other reasons.

Given any arrangement, we can sidestep causes (a) and (b), and focus on whether (c) holds or not. We study two related ideals.  First, we study the ideal given by the intersection of the top dimensional primary components of $J$, which we will call $J^{top}$. The primary components of $J^{top}$ define codimension 2 subschemes of $\mathbb P^n$. We will also study the radicals $\sqrt{J^{top}}$ and $\sqrt{J}$.  The ideal $J^{top}$ is unmixed by definition, and so $\sqrt{J^{top}}$ is also unmixed. But $\sqrt{J}$ is not necessarily  unmixed, unless we add a condition. Consider the following motivating example. If $f = f_1 f_2$, where $f_1$ defines a quadric cone in $\mathbb P^3$ and $f_2$ defines a general plane, then $J^{top} = \sqrt{J^{top}}$ defines a smooth  plane conic, while $\sqrt{J}$ defines the disjoint union of a conic and a point. Of course the first is ACM while the second is not. On the other hand, taking $f_2$ to be a plane through the vertex of $f_1$, $J$ has an embedded component but $\sqrt{J}$ is unmixed, defining only a singular conic (a union of lines).   Note also that  $J^{top}$, $\sqrt{J^{top}}$ and $\sqrt{J}$ are saturated.

So the point  of studying $\sqrt{J^{top}}$ rather than $\sqrt{J}$ in general is that  the Jacobian ideal catches all the singularities of all codimensions, so $\sqrt{J}$ could allow a singularity  that is not embedded in $J^{top}$ to survive.  Under the additional assumption that all $f_i$ are smooth, this is not an issue and we also show that $\sqrt{J}$ defines an ACM scheme. This issue does not arise with hyperplane arrangements.

More precisely, consider a primary decomposition of $J$:
\[
J = \mathfrak q_1 \cap \dots \cap \mathfrak q_r
\]
and let $\mathfrak p_i$ be the associated prime for $q_i$ for $1 \leq i \leq r$. Removing all associated primes of height $>2$, the remaining intersection is well-defined, and we denote by $J^{top}$ this intersection. Similarly, the intersection of all associated primes absorbs all the associated primes of height $>2$, and the resulting ideal is the radical $\sqrt{J}$. When all the $F_i$ have degree 1 (i.e. it is a hyperplane arrangement), this is all we need. In  this paper, though, the singularities of the individual hypersurfaces $F_i$ have to be dealt with, especially when they have codimension 1 in $F_i$. These singularities will survive the process just described. 


The paper \cite{MNS} focused on a study of  the  schemes defined by $J^{top}$ and $\sqrt{J}$ in the case of hyperplane arrangements (i.e. when all the $f_i$ are linear forms), and specifically asked when these schemes are ACM. 
A main result was a condition on $\mathcal A$ which forces the schemes associated to both of these ideals to be ACM.

\begin{theorem}[{\cite[First main theorem]{MNS}}] \label{MNS first main thm-intro}
    Let $\mathcal A$ be a hyperplane arrangement in $\mathbb P^n$ defined by a product $f$ of linear forms. Let $J, \sqrt{J}$ and $J^{top}$ be the ideals defined above. Assume that no linear factor of $f$ is in the associated prime of any two non-reduced components of $J^{top}$. Then both $S/\sqrt{J}$ and $S/J^{top}$ are Cohen-Macaulay.
\end{theorem}


The paper \cite{MNS} also carefully studied the situation when the schemes defined by $\sqrt{J}$ and $J^{top}$ are not ACM,   using primarily certain tools from Liaison Theory applied to hyperplane arrangements. In the case of hyperplane arrangements in $\mathbb P^3$, some results were also obtained on the size and shape of the Hartshorne-Rao module when the ACM property does not hold. This can be interpreted as a measure of the failure to be ACM.

For hyperplanes, it is intuitively clear that any component of the scheme defined by $J^{top}$ depends only on the hyperplanes of $\mathcal A$ containing its support, and this fact was used implicitly in \cite{MNS}. However, we are not aware of a published proof of this fact, and one of our first results in this paper  (Proposition \ref{jac = union}) makes the intuition precise. 

Our focus in this  paper is to extend many of the results from \cite{MNS} to hypersurface arrangements, and in particular \Cref{MNS first main thm-intro}. Here several new complications arise. Among them we note the following. First, a hyperplane is automatically smooth, and so singularities of the arrangement come precisely where hyperplanes of the arrangement meet; this is certainly not the case for hypersurfaces. Second, any two distinct hyperplanes automatically meet in a smooth codimension two complete intersection, but this is not true for hypersurfaces. Third, 
for a hyperplane arrangement $\mathcal A$,  if we choose a codimension two linear component $\Lambda$ of the scheme defined by $\sqrt{J}$ and restrict to the hyperplanes containing $\Lambda$, the Jacobian ideal of this subarrangement is a complete intersection. Passing to hypersurfaces, it becomes much more subtle, but is ``almost true" (see below).

Section \ref{sec: background} sets up the notation and recalls the constructions of Liaison Addition (due to P. Schwartau  \cite{Sw}) and Basic Double Linkage (due to Lazarsfeld and Rao \cite{LR}) that are crucial to our work in this  paper. It does not yet justify the fact that these tools can be applied to Jacobian ideals of hypersurface arrangements. 

Section \ref{one CI} proves a crucial fact for our construction, justifying the ``almost true" assertion made above. More precisely, with  suitable assumptions we show that the  scheme defined by the Jacobian ideal  is still a complete intersection of a certain precise type, but the Jacobian ideal is not saturated (Theorem \ref{sat is CI}). One of the assumptions is that all of the hypersurfaces involved have the same degree, which we show in Section \ref{sect: examples} is crucial.

Section \ref{BDL section} is devoted to establishing that for hypersurface (and hyperplane) arrangements, a version of Liaison Addition and Basic Double Linkage holds as long as suitable assumptions are made. The main new assumption is that if $f = f_1 \cdots f_s$ defines the arrangement, we assume that any two factors meet in a codimension two smooth complete intersection. 
 {\it Note that this  assumption forces the singularities of the individual $f_i$ to be of dimension 0 (or empty) and disjoint from the intersection of any  two $f_i$. However, we do not need the $f_i$ to be smooth except in the case when we discuss  $\sqrt{J}$.} See Lemma \ref{consequence of smooth}.

The main tool is Theorem \ref{LAthm1} (Liaison Addition for Arrangements). Intuitively it gives a very strong statement about the schemes associated to the union of two arrangements, making only the  additional assumption that the intersection of any two factors from one arrangement and one factor from the  other meet in codimension three. The result essentially says that in this situation, Liaison Addition in the  original sense can be applied. As a corollary we also show that Basic Double Linkage in the  original sense can be used for hypersurface arrangements (Corollary  \ref{BDL corollary}). It can be interpreted as showing what happens when ``most" surfaces are added one at a time to a hypersurface arrangement.

Both  of these tools  were fundamental in \cite{MNS} for the case of hyperplane arrangements, but showing that they still apply for hypersurface arrangements requires  suitable assumptions and much  more work. 

Section  \ref{sec: main theorem} gives the main result of this paper, Theorem \ref{cond star implies ACM}. In particular, it establishes the following generalization of \Cref{MNS first main thm-intro}.

\begin{theorem}
    \label{thm:intro-generalization}
Let $\mathcal A$ be a hypersurface arrangement in $\mathbb P^n$ defined by a form $f = f_1\cdots f_s$.  Let $J$ be its Jacobian ideal, and let $J^{top}$ and  $\sqrt{J^{top}}$ be as defined above. 
Assume that any two factors of $f$ meet in a smooth codimension two complete intersection, and that no $f_i$ is in the associated prime of more than one non-reduced component of $J^{top}$. 

Then  the schemes defined by $J^{top}$ and by  $\sqrt{J^{top}}$ are ACM. Furthermore, both  are equidimensional unions of complete intersections, each supported  at a complete intersection defined by some $(f_i, f_j)$.      

 Finally, if in addition each $f_i$ is smooth then also $\sqrt{J}$ defines an ACM scheme.
\end{theorem}


Note that the smoothness assumption in this result is clearly satisfied for any hyperplane arrangement,  and that (as noted above) the smoothness assumption in the last statement is not so restrictive since in any case each $f_i$ has at worst zero-dimensional singularities.

In Section \ref{sect: examples}, we give several examples  to indicate the sorts of problems that arise when one relaxes some of the assumptions in \Cref{thm:intro-generalization}.

One of the applications given in \cite{MNS} was to use Liaison Addition to show that every Buchsbaum even liaison class contains curves that arise both as schemes defined by $J^{top}$ and by $\sqrt{J}$, for suitable hyperplane arrangements $J$. In Section \ref{sec: HR} we discuss to what extent this might be extended to arrangements of hypersurfaces of higher degree using our methods.  In particular, we show that adding a general hypersurface to an arrangement preserves the even liaison class. Thus, it preserves the Cohen-Macaulay property if present originally.

One can view this paper as the next step in a logical sequence of ideas. In the paper \cite{GHM}, the authors studied (among other things) hyperplane arrangements where no three hyperplanes vanish on the same codimension two linear space. One of the results was that the   ideal of the singular locus (which coincides with the radical of the Jacobian)  is always ACM (even if it did not use the language of the Jacobian ideal). The paper \cite{GHMN} took the next step (among other things) of replacing the hyperplanes by hypersurfaces. It was shown that the ACM conclusion continues to hold, but again does not mention the  Jacobian ideal. The paper \cite{MNS} then introduced the study of hyperplane arrangements as described above, but weakened the assumption to allow non-reduced components of the Jacobian ideal. As noted above, the ACM conclusion continued to hold, even if new methods were needed. In this paper we take the next logical step of again replacing hyperplanes by hypersurfaces, and studying the Jacobian ideals. This highly complicates things, but with suitable assumptions the ACM conclusion still holds. In  Section \ref{sect: examples},  one of our goals is to indicate the difficulties of carrying this program on any further.  See for instance Example \ref{ex: twc}.

\section{Background} \label{sec: background}

Let $S = K[x_0,\ldots, x_n]$, where $K$ is algebraically closed of characteristic zero, and let $\mathfrak m = (x_0,\dots,x_n)$, the irrelevant ideal of $S$. Throughout this paper, we sometimes abuse  notation and denote  both a homogeneous polynomial and the hypersurface that it defines by the same capital letter, e.g. $G$. 

\begin{definition} \label{def of ACM}
    
 A subscheme $V$ of $\mathbb P^n$ is {\it arithmetically Cohen-Macaulay} (ACM) if $S/I_V$ is a Cohen-Macaulay ring, where $I_V$ is the saturated homogeneous ideal defining $V$. By abuse of notation, we will sometimes also refer to the ideal $I_V$ as being ACM. When $\dim V \geq 1$, an equivalent condition for $V$ to be ACM is for the Hartshorne-Rao modules of $V$ to all vanish:
\[
H^i(\mathbb P^n, \mathcal I_V(t)) = 0 \ \ \hbox{ for all } t \in \mathbb Z  \hbox{ and all }  1 \leq i \leq \dim V,
\]
where $\mathcal I_V$ is the ideal sheaf of $V$. (See for instance \cite{mig-book}.)
\end{definition}

\begin{notation}
Let $\mathcal A$ be a hypersurface arrangement of $s$ distinct hypersurfaces, i.e. $\mathcal A = \bigcup_{1 \leq i \leq s} G_i$. In some situations, but not all, we will assume that the $G_i$ all have the same degree, $d$. Let $G = \prod_{1 \leq i \leq s} G_i$. Let 
\[
J = \Jac(\mathcal A) = \Jac(G) = \left (\frac{\partial G}{\partial x_0}, \dots , \frac{\partial G}{\partial x_n} \right )
\]
be the Jacobian ideal of $\mathcal A$.  Note that $J$ is $\mathfrak m$-primary if and only if $G$ is smooth (which can only happen when $s=1$).

Now consider a primary decomposition of $J$:
\[
J = \mathfrak q_1 \cap \dots \cap \mathfrak q_r \cap \mathfrak q_i' \cap \dots \cap \mathfrak q_s'
\]
where the $\mathfrak q_i$ are the components of codimension 2 and the $\mathfrak q_i'$ all have codimension $>2$.
Let $\mathfrak p_i$ be the associated prime of $\mathfrak q_i$. 
We set
\[
J^{top}  =  \mathfrak q_1 \cap \dots \cap q_r 
\]
and we note that 
\begin{equation} \label{radicals}
\sqrt{J} =  \mathfrak p_1 \cap \dots \cap \mathfrak p_r 
\end{equation}
Both $J^{top}$ and $\sqrt{J}$ are well-defined, unmixed ideals. We denote by $X^{top}$ and $X^{red}$ the equidimensional schemes defined by $J^{top}$ and $\sqrt{J}$. We also denote by $J^{sat}$ the saturation of $J$ with respect to the irrelevant ideal $\mathfrak m$.
\end{notation}

We now recall some of the tools used in \cite{MNS}, which we reproduce here for convenience.
First we recall the construction of liaison addition. This was introduced by P. Schwartau in his Ph.D.\ thesis \cite{Sw}, which was never published.  The version that we need (and state) is entirely due to Schwartau, 
but we cite \cite{GM4} and \cite{mig-book} as the only convenient sources at this point.

\begin{theorem}[{{\bf (Liaison Addition)} \cite[Corollary 1.6, Theorem 1.3 and Corollary 1.5]{GM4}, \cite[Section 3.2]{mig-book}}] \label{LA}
Let $V_1$ and $V_2$ be locally Cohen-Macaulay, equidimensional codimension two subschemes in $\mathbb P^n$.  Choose homogeneous polynomials $F_1, F_2$ with $d_i = \deg F_i$ so that 
\[
F_1\in I_{V_1} \ \ \ \hbox{ and } \ \ \ F_2 \in I_{V_2}, 
\]
and furthermore $(F_1,F_2)$ is a regular sequence. Let $V$ be the complete intersection scheme defined by $(F_1,F_2)$. Let $I = F_2 I_{V_1} +  F_1 I_{V_2}$ and let $Z$ be the scheme defined by $I$. Then

\begin{itemize}
\item[(i)] If $V_1,V_2$ and $V$ pairwise have no common components then $Z = V_1 \cup V_2 \cup V$ as schemes.

\item[(ii)] $I$ is a saturated ideal.

\item[(iii)] If $h_X(t)$ denotes the Hilbert function of a scheme $X$ then we have
\[
h_Z(t) = h_V(t) + h_{V_1}(t-d_1) +  h_{V_2}(t-d_2).
\]

\item[(iv)] $Z$ is ACM if and only if both $V_1$ and $V_2$ are  ACM. More generally,  we have
\[
\bigoplus_{t \in \mathbb Z} H^i( \mathbb P^n, \mathcal I_Z (t)) \cong \bigoplus_{t \in \mathbb Z} H^i (\mathbb P^n, \mathcal I_{V_1}(t))(-d_2) \oplus \bigoplus_{t \in \mathbb Z} H^i (\mathbb P^n, \mathcal I_{V_2}(t))(-d_1)
\]
as graded $S$-modules, for $1 \leq i \leq n-2$. 

\end{itemize}
\end{theorem}

The following construction, basic double linkage,  can be obtained from liaison addition. In fact it was introduced for curves in $\mathbb P^3$ by Lazarsfeld and Rao \cite{LR}, and different versions have been used in different contexts. We write down only the version that we need, but we note that there is a much more general version, the so-called {\em basic double G-linkage} -- cf. for instance \cite{MN3} Lemma 3.4 and \cite{mig-book} Theorem~3.2.3 and Remark~3.2.4. Our version follows from liaison addition 
by taking $V_2$ to be the empty set and $I_{V_2} = S$.

\begin{proposition} \label{BDL} {\bf (Basic Double Linkage)}
Let $V_1$ be a locally Cohen-Macaulay, equidimensional  codimension two subscheme in $\mathbb P^n$. Choose homogeneous polynomials $F_1, F_2$ with $d_i = \deg F_i$ and $F_1 \in I_{V_1}$ and $F_2 \in S$, such that $(F_1,F_2)$ is a regular sequence. Let $V$ be the complete intersection scheme defined by $(F_1,F_2)$. Let $I = F_2 I_{V_1} + (F_1)$ and let $Z$ be the scheme defined by $I$. Then

\begin{itemize}

\item[(i)] If $V_1$ and $V$ have no common components then $Z = V_1 \cup V$.

\item[(ii)] $I$ is a saturated ideal.

\item[(iii)] $Z$ is ACM if and only if $V_1$ is ACM. More generally, 
\[
\bigoplus_{t \in \mathbb Z} H^i( \mathbb P^n, \mathcal I_Z (t)) \cong \bigoplus_{t \in \mathbb Z} H^i (\mathbb P^n, \mathcal I_{V_1}(t))(-d_2) 
\]
for $1 \leq i \leq n-2$.

\item[(iv)] $Z$ is linked in two steps to $V_1$.

\end{itemize}

\end{proposition}

Finally, we will use the fact that if $V$ is a subvariety of $\mathbb P^n$ of dimension $\geq 2$  and the general hyperplane section of $V$ is ACM then $V$ itself must be ACM. A more general version  can be found in \cite{HU}.

\begin{proposition}[{\cite[Proposition 2.1]{HU}} ] \label{hyperplane sect acm} 
Let $V$ be a locally Cohen-Macaulay, equidimensional closed subscheme of $\mathbb P^n$ and let $F$ be a general homogeneous polynomial of degree $d$ cutting out on $V$ a scheme $Z \subset V \subset \mathbb P^n$. Assume that $\dim V \geq 2$. Then $V$ is ACM if and only if $Z$ is ACM.
\end{proposition}


\section{Arrangements whose base locus is supported on a single smooth complete intersection} \label{one CI}

Given the Jacobian ideal $J$ of a hypersurface arrangement, 
our approach in this paper is to ``build up" the scheme defined by $J^{top}$, and the scheme defined by $\sqrt{J}$, using methods from liaison theory. In the next section we will justify why these methods are valid in our setting of Jacobian ideals. In this section we do something equally basic: we give a careful description of the ``building blocks" that our construction will use. The main result is Theorem \ref{sat is CI}.

\begin{remark} \label{common factor}
Let $F, P \in S$ be a regular sequence of homogeneous polynomials of the same degree $d$, and let $G_i \in [(F,P)]_d$ for $1 \leq i \leq s$. Then for $i \neq j$, $G_i$ and $G_j$ either form a regular sequence or they agree up to scalar multiplication. In particular, they have no non-trivial common factor. This is a standard fact about pencils of hypersurfaces: the base locus of the pencil is determined by any two of the elements of the pencil.
\end{remark}

It is elementary and well-known (see for instance \cite[Remark 3.1]{MNS}) that if $e$ hyperplanes all contain the same codimension two linear variety $\Lambda$, and if $F$ is the product of the corresponding linear forms, then the Jacobian ideal of $F$ is a saturated complete intersection of type $(e-1,e-1)$. In this section we generalize  this to hypersurfaces. See in particular Theorem \ref{sat is CI}.

In this paper we will often make the assumption that any  two hypersurfaces in our arrangement meet in a smooth codimension 2 subvariety, $C$. The next two lemmas give elementary consequences of this assumption.

\begin{lemma} \label{consequence of smooth}
Let $F,G \subset \PP^n$ be hypersurfaces (not necessarily smooth) of arbitrary degree whose intersection is a smooth codimension two complete intersection subvariety, $C$. Then the singular locus of $F$ and of $G$ is zero-dimensional and disjoint from $C$.
\end{lemma}

\begin{proof}
Without loss of generality suppose that $F$ has a singular locus, $D$, whose dimension is $\geq 1$. Then $D \cap G$ is non-empty. If $P \in D \cap G$ then $P \in D \subset F$  and $P \in G$ by construction, so $P \in C$. The contradiction then comes from an examination of the Jacobian matrix of $(F,G)$ to see that $C$ cannot be smooth.
\end{proof}

\begin{lemma} \label{1st lemma}
Let $F, P \subset \PP^n$ be smooth hypersurfaces of degree $d$ whose intersection is a smooth codimension two complete intersection subvariety, $C$. Let $G_1,G_2 \in [(F,P)]_d$. Assume that $G_1$ and $G_2$ define different hypersurfaces.
Then $G_1$ and $G_2$ meet transversally at every point of $C$. 
\end{lemma}

\begin{proof}
Note that $(G_1,G_2) = (F,P)$. 
Since $C$ is smooth, locally at every point $Q$, $(F,P)_Q$ defines exactly $C$, which is smooth, so the hypersurfaces $F$ and $P$ must meet transversally along $C$ (see also Remark \ref{common factor}).
\end{proof}

 We also need the following strengthening of Theorem 3.3 of \cite{GHMN} in the case of codimension 2. It removes the assumption that the intersection of any 3 of the hypersurfaces has codimension 3. In fact, it is a simple application of basic double linkage. But note that the result is not a complete intersection in general -- this should be contrasted with Theorem \ref{sat is CI}.

\begin{lemma} 
     \label{lem:star config is power}
Let $F,P \in S$ be  a regular sequence of homogeneous polynomials of degree $d$  defining a smooth complete intersection $C$.  For $1 \leq i \leq s$, let $G_i \in [(F,P)]_d$ be forms defining $s$ distinct hypersurfaces of $\PP^n$.   Let $G = \prod_{i=1}^s G_i$. Then the ideal
\[
\left ( \frac{G}{G_1}, \frac{G}{G_2}, \dots, \frac{G}{G_s} \right )
\]
is the saturated ideal of an ACM codimension two subscheme of $\mathbb P^n$ supported on $\Proj S/(F, P)$, of degree $\binom{s}{2} \cdot d^2$. In fact, there is an equality of ideals 
\[
(F,P)^{s-1}  = \left ( \frac{G}{G_1}, \frac{G}{G_2}, \dots, \frac{G}{G_s} \right ).
\]
\end{lemma}

\begin{proof} 
We begin by establishing the first claim.  Denote by $C_s$ the scheme defined by $(\frac{G}{G_1}, \dots, \frac{G}{G_s})$.
By Remark \ref{common factor},  no two $G_i$ share a common factor. 
Notice that $C$ is also defined by $G_1$ and $G_2$ since the $G_i$ also have degree $d$, so $(F,P) = (G_1,G_2)$.  We start with $s=2$, and set $C_2 = C$, and consider the product $G = G_1 G_2$. The ideal of $C_2$ is 
\[
I_{C_2} = (G_2,G_1) = \left ( \frac{G}{G_1}, \frac{G}{G_2} \right ).
\] 
The degree of $C_2$ is $\binom{2}{2} \cdot d^2$. 

Now we construct $C_3$. Notice that $G_3$ and $G_1 G_2$ form a regular sequence. Since $G_1 G_2 \in (G_1,G_2)$, the ideal
\[
I_{C_3} = G_3 \cdot I_{C_2} + (G_1 G_2) = (G_1G_3, G_2 G_3, G_1 G_2)
\]
defines a non-reduced ACM codimension two subvariety $C_3$ of degree 
\[
\deg C_3 = \deg C_2 + \deg(G_3, G_1 G_2) = d^2 + d\cdot (2d) = \binom{3}{2} \cdot d^2.
\]

Now we proceed by induction. Let $G_1, \dots, G_s$ be a set of hypersurfaces of degree $d$  satisfying  our hypotheses and assume that the result is true for $G_1,\dots,G_{s-1}$.  Consider the ideal
\[
I_{C_s} = G_s \cdot I_{C_{s-1}} + G_1 \cdots G_{s-1}.
\]
Again, $G_s$ and $G_1\cdots G_{s-1}$ form a regular sequence. By basic double linkage, this is the saturated ideal of an ACM codimension two subscheme $C_s$ in $\mathbb P^n$. Its degree is
\[
\deg C_s = \deg C_{s-1} + d \cdot \deg(G_1\cdots G_{s-1}) = \binom{s-1}{2} \cdot d^2 + d \cdot (s-1)d
\]
which gives the desired result.

It remains to verify the assertion that $C_s$ is supported on $C$ for all $s$. This follows because $C$ is a complete intersection of type $(d,d)$ and each $G_i \in I_C$ has degree $d$, so this is essentially Remark \ref{common factor}.

Finally, we show the second claim. 
Indeed, each $G_i$ is a linear combination $a_i F + b_i P$ so the inclusion $\supseteq$ is clear. Then both ideals are unmixed saturated ideals of schemes of the same degree, so the reverse inclusion is immediate. 
\end{proof}

Our next goal is to show that if $C$ is a smooth complete intersection subvariety of codimension two of type $(d,d)$, and if $G_i$ ($1 \leq i \leq s$) are all elements of $[I_C]_d$, then the product $G = \prod_{i=1}^s G_i$ has the property that $Jac( G)^{sat}$ is a complete intersection. We first say very precisely which complete intersection it will be, and then we prove the asserted fact. This statement is true without any smoothness assumptions. 

\begin{proposition} 
   \label{lem:ci automatic}
Let $F, P \in S$ be  a regular sequence of homogeneous polynomials of degree $d$.  For $i = 1,\ldots, s$ with $s \ge 2$, let $G_i = a_i F + b_i P \in [(F,P)]_d$ (with scalars $a_i, b_i \in K$) be forms defining $s$ distinct hypersurfaces of $\PP^n$.   Set $G = \prod_{i=1}^s G_i$. Then the two forms 
\begin{align*}
H_1  =  \displaystyle \sum_{i=1}^s a_i \frac{G}{G_i} \quad \text{ and } \quad
H_2   =  \displaystyle \sum_{i=1}^s b_i \frac{G}{G_i}  
\end{align*}
form a regular sequence. 
\end{proposition} 

\begin{proof}
We consider first a special case, where $F$ and $P$ are two variables. It will be convenient to use new notation. Let $R = k[x, y]$ be a polynomial ring in two variables. For $i = 1,\ldots, s$, let $\ell_i = a_i x + b_i y$ be linear forms definining $s$ distinct hyperplanes.  Define $L = \ell_1 \cdots \ell_s$, 
\[
h_1 = \displaystyle \sum_{i=1}^s a_i \frac{L}{\ell_i} \quad \text{ and } \quad h_2 = \displaystyle \sum_{i=1}^s b_i \frac{L}{\ell_i}. 
\]
By the choice of the linear forms, we know $(\ell_i, \ell_j) = (x, y)$ whenever $i \neq j$. 
We use induction on $s$ to show that $h_1$ and $h_2$ form an $R$-regular sequence. Let $s = 2$. Since 
\[
(\ell_1, \ell_2) = (a_1 x + b_1 y, a_2 x + b_2 y) = (x, y), 
\]
the matrix $\begin{bmatrix}
a_1 & a_2 \\
b_1 & b_2
\end{bmatrix}$ is invertible. It follows that  
\[
(h_1, h_2) = (a_1 \ell_2 + a_2 \ell_1,  b_1 \ell_2 + b_2 \ell_1) = (\ell_1, \ell_2)
\] 
 has codimension two, as desired. 

Let $s \ge 3$. Assume that $h_1$ and $h_2$ have a greatest common divisor, $g$, of positive degree. Note that 
\[
\begin{bmatrix}
\ell_1 & \cdots & \ell_s 
\end{bmatrix} = \begin{bmatrix}
x & y 
\end{bmatrix} \cdot \begin{bmatrix}
a_1 & \ldots & a_s \\
b_1 & \ldots & b_s
\end{bmatrix}
\] 
and 
\[
\begin{bmatrix}
h_1 \\
h_2 
\end{bmatrix}  = \begin{bmatrix}
a_1 & \ldots & a_s \\
b_1 & \ldots & b_s
\end{bmatrix}
\cdot 
\begin{bmatrix}
\frac{L}{\ell_1} \\
\vdots \\
\frac{L}{\ell_s}
\end{bmatrix}. 
\] 
Thus, we obtain
\[
x h_1 + y h_2 =  \begin{bmatrix}
x & y 
\end{bmatrix} \cdot \begin{bmatrix}
h_1 \\
h_2 
\end{bmatrix}  = \sum_{i = 1}^s \ell_i \frac{L}{\ell_i} = s L. 
\]
Hence $g$ divides $L = \ell_1 \cdots \ell_s$. Since $\ell_1, \ldots, \ell_s$ are irreducible and pairwise coprime, possibly  after re-indexing, we may assume that $g =  \ell_1 \cdots \ell_k$ for some $k$ with $1 \le k \le s-1$. In particular, none of the forms $\ell_{k+1},\ldots, \ell_s$ divides $g$. Using that $h_1 =  \sum_{i=1}^k a_i \frac{L}{\ell_i} + \sum_{i= k+1}^s a_i \frac{L}{\ell_i}$, one gets 
\[
h_1 = \ell_{k+1} \cdots \ell_s \cdot h_1' +  \ell_1 \cdots \ell_k \cdot \widetilde{h}_1 = 
\ell_{k+1} \cdots \ell_s \cdot h_1' +  g \cdot \widetilde{h}_1
\]
with some $\widetilde{h}_1 \in R$ and $h_1' =  \sum_{i=1}^k a_i \frac{g}{\ell_i}$. Similarly, we see
\[
h_2 = \ell_{k+1} \cdots \ell_s \cdot h_2' +  g \cdot \widetilde{h}_2
\] 
with some $\widetilde{h}_2 \in R$ and $h_2' =  \sum_{i=1}^k b_i \frac{g}{\ell_i}$. Since $g$ divides $h_1$ and $h_2$ and is relatively prime to $ \ell_{k+1} \cdots \ell_s$, we conclude that $g$ divides $h_1'$ and $h_2'$. However, the induction hypothesis applied to $\ell_1,\ldots, \ell_k$  gives that $h_1'$ and $h_2'$ form a regular sequence if $k \ge 2$, a contradiction. If $k = 1$, i.e., $g = \ell_1$, we have $h_1  = a_1 
\ell_{2} \cdots \ell_s  +  \ell_1 \cdot \widetilde{h}_1$ and $h_2  = b_1 
\ell_{2} \cdots \ell_s  +  \ell_1 \cdot \widetilde{h}_2$.  Since $a_1$ and $b_1$ cannot be both zero, it follows that $g = \ell_1$ does not divide both $h_1$ and $h_2$, a contradiction. Thus, we have shown that $h_1$ and $h_2$ form an $R$-regular sequence. 

Second, we discuss the general case. 
Consider  the $k$-algebra homomorphism $\ffi \colon R \to S$, defined by $x \mapsto F$ and $y \mapsto G$. 
It is flat because $F, G$ is a regular sequence by assumption. Hence $H_1 = \ffi (h_1)$ and $H_2 = \ffi (h_2)$ form an $S$-regular sequence by \cite[Theorem 15.1]{M}. 
\end{proof}

We are ready to establish the main result of this section.

\begin{theorem} \label{sat is CI} 
Consider a smooth complete intersection $Z \subset \PP^n$ of dimension $n-2$ cut out by two hypersurfaces of degree $d$. Let $ G_1,\ldots,G_s  \in [I_Z]_d$ be forms defining $s$ distinct  smooth hypersurfaces of $\PP^n$.   Set $G = \prod_{i=1}^s G_i$.
Then the saturation of the Jacobian ideal $J = \Jac(G)$ is a complete intersection. In particular, $J^{top} = J^{sat}$ is ACM, and $\sqrt{J} = I_Z$ is ACM. 

More, precisely, 
$J^{sat} = (H_1,H_2)$ is a complete intersection of type $((s-1)d, (s-1)d)$, where $H_1$ and $H_2$ are defined as in \Cref{lem:ci automatic} with $I_Z = (F, P)$ and $ G_i = a_i F + b_i P$.
\end{theorem}

\begin{proof} 
For any variable $x_k$ of $S$, one computes
\begin{align} 
\frac{\partial G}{\partial x_k}  & = \sum_{i=1}^s \left [ a_i \frac{\partial F}{\partial x_k} + b_i \frac{\partial P}{\partial x_k} \right ] \frac{G}{G_i} \nonumber \\[2pt]
 & =  \left ( \frac{\partial F}{\partial x_k} \cdot \sum_{i=1}^s a_i \frac{G}{G_i} \right ) + \left ( \frac{\partial P}{\partial x_k} \cdot \sum_{i=1}^s b_i \frac{G}{G_i} \right ) \in (H_1,H_2).    \label{eq:derivatives of G} 
\end{align}
We conclude that $\Jac(G) \subseteq (H_1,H_2)$. Since a regular sequence generates a saturated ideal, and saturation preserves inclusions, we get by \Cref{lem:ci automatic} 
\[
\Jac(G)^{sat}  \subseteq (H_1, H_2)^{sat} = (H_1,H_2). 
\]

Next, we establish the reverse inclusion. 
Since $Z$ is smooth  and any two of the $G_i$ meet at $Z$ and nowhere else (see \Cref{common factor}), the singular locus (as a set) of the scheme defined by $G$ is $Z$, and so the scheme defined by $\Jac(G)$ is supported precisely on~$Z$. 
Using also that $(H_1, H_2) \subseteq \left ( \frac{G}{G_1}, \frac{G}{G_2}, \dots, \frac{G}{G_s} \right )$ and \Cref{lem:star config is power}, we conclude that both $\Jac (G)$ and $(H_1, H_2)$ define subschemes supported precisely at $Z$. 

For any form $E \in S$, let us denote by $\nabla E$ the Jacobian matrix $\begin{bmatrix} 
E_{x_0} & \ldots & E_{x_n} 
\end{bmatrix}$ of $E$, where $E_{x_k}$ stands for $\frac{\partial E}{\partial x_k}$. Observe that Equation \eqref{eq:derivatives of G} can be written as 
\begin{equation}
    \label{eq:matrix version}
\nabla G = \begin{bmatrix}
H_1 & H_2 
\end{bmatrix} \cdot \begin{bmatrix}
\nabla F \\
\nabla P
\end{bmatrix}. 
\end{equation}
We also have 
\[
\begin{bmatrix}
H_1 \\
H_2 
\end{bmatrix}  = \begin{bmatrix}
a_1 & \ldots & a_s \\
b_1 & \ldots & b_s
\end{bmatrix}
\cdot 
\begin{bmatrix}
\frac{G}{G_1} \\
\vdots \\
\frac{G}{G_s}
\end{bmatrix}. 
\] 

Consider now the Jacobian matrix $\begin{bmatrix}
\nabla F \\
\nabla P
\end{bmatrix}$ of $(F, P)$. By assumption, $(F, P)$ defines a smooth subscheme $Z$ of $\PP^n$. Hence, the ideal generated by the 2-minors of the Jacobian matrix does not vanish at any point of $Z$. Consider any such point $Q$ of $Z$. Thus, there is a minor, say, $F_{x_i} P_{x_j} - F_{x_j} P_{x_i}$, that does not vanish at $Q$, where $i$ and $j$ may depend on the choice of $Q$. In other words, considered as a matrix with entries in  the localization $S_{I_Q}$, the matrix $\begin{bmatrix}
F_{x_i} & F_{x_j} \\
P_{x_i} & P_{x_j} \\
\end{bmatrix}$ is invertible. Hence, using 
\[ 
\begin{bmatrix}
G_{x_i}  &
G_{x_j} 
\end{bmatrix}
 = 
\begin{bmatrix}
H_1 & H_2 
\end{bmatrix}
 \cdot 
\begin{bmatrix}
F_{x_i} & F_{x_j} \\
P_{x_i} & P_{x_j} \\
\end{bmatrix}, 
\]
one obtains 
\[
 (H_1, H_2)_{I_Q} = (G_{x_i}, G_{x_j})_{I_Q} \subset \Jac (G)_{I_Q}. 
\]
We conclude that, for every point $Q$ of $Z$, one has  $(H_1, H_2)_{I_Q}  \subset \Jac (G)_{I_Q}$. We observed above that  the schemes defined by $(H_1, H_2)$ and $\Jac (G)$ are both supported precisely at $Z$. It follows that $(H_1,H_2)^{sat} \subseteq \Jac(G)^{sat}$, and so $(H_1,H_2) \subseteq \Jac(G)^{sat}$. This is the desired reverse inclusion, and we are done.
\end{proof} 

In Section \ref{sect: examples} we will discuss examples showing that the conclusions of Theorem \ref{sat is CI} may fail if one relaxes some of the assumptions. 

As an immediate consequence of \Cref{sat is CI}, we have the  well-known result mentioned at the beginning  of this section:

\begin{corollary} \label{plane ci}
Let $\mathcal A$ be a hyperplane arrangement defined by a product $L_1\cdots L_s$, where each $L_i$ is in $I_{\Lambda}$ for some fixed codimension two linear variety $\Lambda$ and $s \geq 2$. Let $J$ be the Jacobian ideal of $\mathcal A$. Then $J$ is saturated but not reduced (supported on $\Lambda$), and is the complete intersection of two forms of degree $s-1$. In particular, $S/J$ is Cohen-Macaulay.
\end{corollary}

\begin{proof}
This is the case $d=1$ of \Cref{sat is CI}. The only additional point needed is that $J$ is already saturated here (unlike in \Cref{sat is CI}). Indeed, setting $I_\Lambda = (x,y)$ (by change of variables) we see that the Jacobian ideal $J$ is generated by two forms of degree $s-1$, as is its saturation, so they are equal.
\end{proof}

We record some consequences, which may be of independent interest. They 
say that certain ideals of minors have the expected (maximal) codimension and so are perfect ideals. We refer to a smooth polynomial as a form in $S$ that defines a smooth hypersurface on $\PP^n$.

\begin{corollary}
    \label{cor:amusing cor}

Let $F$ and $P$ be hypersurfaces of $\PP^n$ of the same degree meeting in a smooth $(n-2)$-dimensional subscheme. 
Then one has: 
\begin{itemize}

\item[(a)] The ideal generated by the 2-minors of the $2 \times (n+1)$ Jacobian matrix $\begin{bmatrix}
\nabla F \\
\nabla P
\end{bmatrix}$ has codimension $n$.

\item[(b)] The ideal generated by the 2-minors of the $2 \times (n+2)$  matrix $\begin{bmatrix}
\nabla F  & H_2 \\
\nabla P & -H_1
\end{bmatrix}$ has codimension $n+1$, where $H_1, H_2$ are the polynomials defined in \Cref{lem:ci automatic} using any smooth polynomials $G_1,\ldots,G_s \in (F, P)$ of degree $\deg P$.

\end{itemize}
    
\end{corollary} 

\begin{proof}
Claim (a) follows from (b). Observe that, by Bertini's Theorem, any general form in $(F,P)$ of degree $\deg P$ defines a smooth hypersurface. Thus, suitable smooth hypersurfaces $G_1,\ldots,G_s$ do exist. Setting $G = G_1 \cdots G_s$, \Cref{sat is CI} gives $\Jac (G)^{sat} = (H_1, H_2)$, that is, $[\Jac (G)]_k  = [(H_1, H_2)]_k$ for any integer $k \gg 0$. Hence, for any polynomials $s_1, s_2 \in [R]_k$ with $k \gg 0$, there are homogeneous polynomials $r_0,\ldots,r_n \in R$ such that 
\[
s_1 H_1 + s_2 H_2 = \sum_{i = 0}^n r_i G_{x_i}.  
\]
Using also Equation \eqref{eq:matrix version}, this gives 
\[
\begin{bmatrix} 
H_1 & H_2 
\end{bmatrix} 
\cdot
\begin{bmatrix}
s_1 \\ 
s_2
\end{bmatrix}
= 
\nabla G \cdot 
\begin{bmatrix}
r_0\\
\vdots\\
r_n
\end{bmatrix} 
= 
\begin{bmatrix}
H_1 & H_2 
\end{bmatrix} \cdot \begin{bmatrix}
\nabla F \\
\nabla P
\end{bmatrix} 
\cdot 
\begin{bmatrix}
r_0\\
\vdots\\
r_n
\end{bmatrix}. 
\]
 Thus, we obtain 
 \[
 0 = 
 \begin{bmatrix} 
H_1 & H_2 
\end{bmatrix} 
\cdot
\left ( \begin{bmatrix}
s_1 \\ 
s_2
\end{bmatrix} - 
\begin{bmatrix}
\nabla F \\
\nabla P
\end{bmatrix} 
\cdot 
\begin{bmatrix}
r_0\\
\vdots\\
r_n
\end{bmatrix} \right ). 
 \]
Since the syzygy module of the ideal $(H_1, H_2)$ is generated by 
$\begin{bmatrix}
H_2 \\
-H_1
\end{bmatrix}$, there is some form $q \in R$ such that 
\[
\begin{bmatrix}
s_1 \\ 
s_2
\end{bmatrix} - 
\begin{bmatrix}
\nabla F \\
\nabla P
\end{bmatrix} 
\cdot 
\begin{bmatrix}
r_0\\
\vdots\\
r_n
\end{bmatrix}
= 
q \cdot 
\begin{bmatrix}
H_2 \\
-H_1
\end{bmatrix}. 
\]
This can be rewritten as 
\[
\begin{bmatrix}
s_1 \\ 
s_2
\end{bmatrix}  = \begin{bmatrix}
\nabla F  & H_2 \\
\nabla P & -H_1
\end{bmatrix} 
\cdot 
\begin{bmatrix}
r_0\\
\vdots\\
r_n \\
q
\end{bmatrix}. 
\]
Consider now the $R$-module homomorphism $\ffi \colon R^{n+2} \to R^2$ with $v \mapsto A v$, where
 $A = \begin{bmatrix}
\nabla F  & H_2 \\
\nabla P & -H_1
\end{bmatrix}$. 
The previous equation implies that the module  $M = \coker \ffi$ has finite length, which means that $\codim (\Ann_R (M)) = n+1$. 
Hence (see, e.g., \cite[Proposition 20.7]{E-book}),  the Fitting ideal $I_2 (A)$ that is generated by the $2$-minors of $A$ has also codimension $n+1$,  
as claimed. 
\end{proof}



\section{Decompositions of Hypersurface Arrangements} \label{BDL section}


In this  section, we justify using the liaison tools on the Jacobian ideals for hypersurface arrangements. Basically, we show that the scheme we can produce by a certain scheme-theoretic union of complete intersections (and which, in suitable situations, we construct using the liaison tools) is exactly the same as the scheme defined by the top dimensional part of the Jacobian ideal. 

We start with hyperplanes. In the following result we denote by $L_i$ a hyperplane and by $\ell_i$ the linear form defining it (up to scalars). We first prove a result that was used implicitly in \cite{MNS}. We provide an argument since we are not aware of a reference in the literature. 

\begin{proposition} \label{jac = union}
Let $\mathcal A = \bigcup_{i=1}^s L_i \subseteq \mathbb P^n$ ($s \geq 2$) be an arbitrary hyperplane arrangement. Then the top dimensional part of the scheme defined by $\Jac(\ell_1 \cdots \ell_s)$ is a union of complete interections. Each of these complete intersections is supported at some linear space $\Lambda = L_i \cap L_j$ and equal to the  scheme defined by $\Jac(f)$, where $f$ is the product of the linear forms $\ell_i$ in $I_\Lambda$. 
\end{proposition}

\begin{proof}

Without loss of generality assume that $x = x_0$ and $y = x_1$ define two of the hyperplanes of $\mathcal A$. Let $\Lambda = \mathbb V( (x,y))$. Write 
\[
\ell_1 \cdots \ell_s =  \underbrace{ \ell_1 \cdots \ell_d}_f \cdot \underbrace{ \ell_{d+1} \cdots \ell_s}_g
\]
with $\ell_1, \dots, \ell_d \in (x,y)$ and $g \notin (x,y)$. 
Thus we get
\[
\Jac(fg) = (f_x g + f g_x, f_y g + f g_y, f \cdot g_{x_i} \ | \ 2 \leq i \leq n ).
\]
Since we are working over a field of characteristic zero, Euler's formula gives 
\[
d \cdot f = xf_x + y f_y.
\]
It follows that 
\[
\begin{array}{rcl}
f_x g + f g_x & = & f_x (g + \frac{1}{d} x g_x) + f_y ( \frac{1}{d} y g_x ) \ \ \ \hbox{ and } \\ \\
f_y g + f g_y & = & f_x (\frac{1}{d} x g_y) + f_y (g+ \frac{1}{d} y g_y).
\end{array}
\]
Rewrite this as 
\[
\left [
\begin{array}{c}
(fg)_x \\
(fg)_y
\end{array}
\right ]
= 
\left [
\begin{array}{cc}
g + \frac{1}{d} x g_x & \frac{1}{d} y g_x \\
\frac{1}{d} x g_y & g + \frac{1}{d} y g_y 
\end{array} 
\right ]
\left [
\begin{array}{c}
f_x \\
f_y
\end{array}
\right ].
\]
The determinant of the $2 \times 2$ matrix is
\[
det = g^2 + h
\]
with $h \in (x,y)$. Since $g \notin (x,y)$ we conclude $det \notin (x,y)$. Hence $det$ is a unit in the localization of $R = k[x_0,\dots, x_n]$ at $(x,y)$. This implies
\[
(f_x, f_y)_{(x,y)} = ((fg)_x, (fg)_y)_{(x,y)} \subseteq \Jac(fg)_{(x,y)}.
\]
By \Cref{plane ci}, $\Jac(f)$ is a complete intersection of type $(d-1,d-1)$.  Since clearly $\Jac(fg) \subset \Jac(f) = (f_x, f_y)$, we obtain
\[
(f_x, f_y)_{(x,y)} = \Jac(fg)_{(x,y)}
\]
as claimed.
\end{proof}

As a consequence we get the following result. We are not aware of a published proof of the first equality.

\begin{corollary}  \label{reduced case}
Let $\mathcal A = \bigcup_{i=1}^s L_i $ be a hyperplane arrangement in $\mathbb P^n$ and assume further that any three of the hyperplanes intersect in a codimension 3 linear space. Let $L = \ell_1 \cdots \ell_s$. Then
\[
\Jac(\ell_1 \cdots \ell_s)^{top} = \bigcap_{i < j} (\ell_i, \ell_j) = \left ( \frac{L}{\ell_1}, \dots, \frac{L}{\ell_s} \right )
\]
is Cohen-Macaulay.
\end{corollary}

\begin{proof}
What is being asserted here is the first equality.  For the second equality and the Cohen-Macaulayness see \cite{GHM} or \cite{PS}. The first equality is immediate from \Cref{jac = union}.
\end{proof}

In \cite{GHM} the Cohen-Macaulayness was proved by observing that the scheme defined by the union of codimension two linear varieties  defined by $\bigcap (\ell_i, \ell_j)$ could be produced via a sequence of basic double links starting from a complete intersection. What is new here is that this process gives the top dimensional part of the scheme defined by the Jacobian ideal. That is, this result justifies the use of the liaison methods as a tool to describe schemes arising from the top dimensional part of Jacobian ideals of hyperplane arrangements. This was used implicitly in \cite{MNS}.

As noted in the introduction, there are several complications to extending this result to hypersurfaces, and in the rest of the section we develop the machinery and prove the generalizations. 

An immediate consequence of Theorem \ref{sat is CI}  is that if $f$ and $g$ are smooth forms of the same degree $d$ defining a smooth complete intersection, then the  saturation of the Jacobian ideal $\Jac(fg)$ is a complete intersection of type $(d,d)$. Our next statement (Proposition \ref{two factors})  generalizes this result, allowing the forms to have different degrees.   We recall (Lemma \ref{consequence of smooth}) that the assumption that $(f,g)$ defines a smooth complete intersection $Z$ forces $f$ and $g$ both to have at most 0-dimensional singularities which are disjoint from $Z$.

\begin{proposition} \label{two factors}
Assume that $(f,g)$ defines a smooth complete intersection, $Z$. We do not require $\deg(f) = \deg(g)$. 
Then $\Jac(fg)^{top} = (f,g)$.
\end{proposition}

\begin{proof}

We know by smoothness that
\[
\sqrt{\Jac(fg)^{top}} = (f,g)
\]
and in particular that  the top dimension part of $Sing(fg)$ is $Z$. (Note that either $f$ or $g$ could have singularities of higher codimension.) In particular, the equation gives  $\Jac(fg)^{top} \subseteq (f,g)$, and so we only have to prove the reverse inclusion.

For any $0 \leq i \leq n$ we have
\begin{equation*} 
\frac{\partial}{\partial x_i} (fg) = \frac{\partial}{\partial x_i} (f) \cdot g + f \cdot \frac{\partial}{\partial x_i} (g) \in (f,g).
\end{equation*}
We rewrite these equation as follows:
\[
\nabla (fg) = [g \ \ f]
\left [
\begin{array}{c}
\nabla f \\
\nabla g
\end{array} 
\right ].
\]
Since $Z$ is smooth, for any point $Q \in Z$ there is a 2-minor, say 
\[
\det \left [
\begin{array}{cc}
\frac{\partial f}{\partial x_i} & \frac{\partial f}{\partial x_j} \\
\frac{\partial g}{\partial x_i}  & \frac{\partial g}{\partial x_j} 
\end{array}
\right ],
\]
that does not vanish at $Q$. 
So in the localization $S_{I_Q} = K[x_0,\dots,x_n]_{I_Q}$ we have
\[
\left [ \frac{\partial (fg)}{\partial x_i} \ \ 
\frac{\partial (fg)}{\partial x_j} \right ] = 
[g \ \ f] \cdot  A
\]
for some invertible matrix $A$. Thus
\[
\left [ \frac{\partial (fg)}{\partial x_i} \ \ 
\frac{\partial (fg)}{\partial x_j} \right ] \cdot A^{-1} = [g \ \ f], 
\]
and so $(f,g) \subseteq (\frac{\partial (fg)}{\partial x_i}, \frac{\partial (fg)}{\partial x_j})$. Hence
\[
(f,g)_{I_Q} = \Jac(fg)_{I_Q} \ \ \ \hbox{ for any } Q \in Z.
\]
Therefore 
\[
(f,g)_{I_Z} = \Jac(fg)_{I_Z} = (\Jac(fg)^{top})_{I_Z}. 
\]
In particular, the unique top dimensional component of $\Jac(fg)$ is $I_Z$.
\end{proof}

Our next goal is to show that under reasonable assumptions, adding a hypersurface to an arrangement does not affect the top-dimensional primary components that were already  there for the original arrangement. For a homogeneous ideal $I$ of $S$, we denote by $\mathbb{V} (I)$ the variety defined by $I$ in $\PP^n$.

\begin{proposition} \label{compare assoc primes}
Let $f = f_1 \cdots f_s$ ($s \geq 2$) and assume that each $f_i$ is smooth. Let $g$ be a smooth homogeneous polynomial such that $\hbox{codim} (f_i, f_j, g) = 3$ for each $i \neq j$. Let $\mathfrak p \in Ass(S/\Jac(f)^{top})$.  Then $\mathfrak p \in Ass(S/\Jac(fg)^{top})$. 

Furthermore, let $\mathfrak q_1, \mathfrak q_2$ be the $\mathfrak p$-primary components of $\Jac(f)$ and of $\Jac(fg)$, respectively. Then $\mathfrak q_1 = \mathfrak q_2$. 
\end{proposition}

\begin{proof}
Since $\Jac(fg) \subseteq \Jac(f)$ we have that $\Jac(fg)_{\mathfrak p} \subseteq \Jac(f)_{\mathfrak p}$. Consider any point $Q \notin \mathbb V(g)$,  and without loss of generality set $I_Q = (x_1,\dots,x_n)$.  We first prove the following claim.

\vspace{.2in}
\noindent \underline{Claim}: {\it $\nabla (fg)^T = M \cdot (\nabla f)^T$ for some matrix $M$ with $\det(M) \notin I_Q$, where $(-)^T$ represents the transpose of $(-)$.}

\vspace{.1in}

To see this, as before we write $F_x$ for the partial derivative of a polynomial $F$ with respect to a variable $x$. Assume $\deg (f) = d$. Then we have

\begin{itemize}

\item $df = x_0 f_{x_0} + \dots + x_n f_{x_n}$, and so
\[
f = \frac{1}{d} x_0 f_{x_0} + \dots + \frac{1}{d} x_n f_{x_n};
\]

\item $\Jac(fg) = (f_{x_0} g + f g_{x_0}, \dots, f_{x_n} g + f g_{x_n})$. 

\end{itemize}

Thus
\[
\begin{bmatrix}
(fg)_{x_0} \\
\vdots \\
(fg)_{x_n}
\end{bmatrix}
= 
\left [
\begin{array}{cccccccccc}
g + \frac{1}{d} x_0 g_{x_0} & \frac{1}{d} x_1 g_{x_0} & \dots &  \frac{1}{d} x_n g_{x_0} \\
\frac{1}{d} x_0 g_{x_1} & g +  \frac{1}{d} x_1 g_{x_1} & \dots &  \frac{1}{d} x_n g_{x_1} \\
&& \vdots \\
\frac{1}{d} x_0 g_{x_n} &  \frac{1}{d} x_1 g_{x_n} & \dots & g +  \frac{1}{d} x_n g_{x_n} 
\end{array}
\right ]
\cdot 
\left [
\begin{array}{c}
f_{x_0} \\
f_{x_1} \\
\vdots \\
f_{x_n}
\end{array}
\right ].
\]
Denote by $M$ the $(n+1) \times (n+1)$ matrix in the above equation. 
We have assumed that $g \notin I_Q = (x_1,\dots,x_n)$. Thus 
\[
\det (M) = (g + \frac{1}{d} x_0 g_{x_0}) g^{n-1} + h \ \ \ \hbox{ where } h \in I_Q.
\]
But
\[
g = \frac{1}{\deg(g)} \sum_i x_i g_{x_i}.
\]
Hence 
\[
0 \neq g \ \hbox{mod }  I_Q = \frac{1}{\deg(g)} x_0 g_{x_0} \ \hbox{mod } I_Q ,
\]
and so
\[
(g + \frac{1}{d} x_0 g_{x_0}) \ \hbox{mod } I_Q = \underbrace{\left ( \frac{1}{\deg(g)} + \frac{1}{d} \right )}_{\neq 0} \underbrace{x_0 g_{x_0}}_{\neq 0} \ \hbox{mod } I_Q \neq 0.
\]
We obtain
\[
\nabla (fg)^T = M \cdot (\nabla f)^T \ \ \ \hbox{ with } \det(M) \notin I_Q,
\]
and this concludes the proof of the claim. 

\smallskip

Therefore, for any point $Q \in \PP^n$, we have in $S_{I_Q}$ that $\Jac(f)_{I_Q} \subseteq \Jac(fg)_{I_Q}$, and so 
\begin{equation}
  \label{eq:localizations equal} 
\Jac(f)_{I_Q} = \Jac(fg)_{I_Q} \text{ if } Q \in \mathbb{P}^n \backslash \mathbb{V}(g).
\end{equation}

Let ${\mathfrak p, \mathfrak q_1, \mathfrak q_2}$ be as in the statement of this proposition. The assumption $\codim (f_i, f_j, g) = 3$ if $i \ne j$ implies that there is point in $Q$ in $\mathbb{V} (\mathfrak{p})$ that is not in $\mathbb{V}(g)$. Localizing Equality \eqref{eq:localizations equal} at $\mathfrak{p}$, we obtain $\Jac (f)_{\mathfrak{p}} = \Jac (fg)_{ \mathfrak{p}}$. As $\mathfrak q_1$ and  $\mathfrak q_2$ are minimal components  of their respective ideals, this gives the desired equality $\mathfrak q_1 = \mathfrak q_2$. 
\end{proof}

After all this preparation, we are ready to state one of the main tools of this paper.

\begin{theorem} \label{LAthm1} {\bf (Liaison Addition for Arrangements.)} 

 Let $\mathcal A_{fg} = \mathcal A_f \cup \mathcal A_g \subset \PP^n$ be a hypersurface arrangement, where $f = f_1 \cdots f_s$ and $g = g_1 \cdots g_t$, such that 

\begin{enumerate}[label=(\roman*)]


\item Any two distinct hypersurfaces in $\mathcal A_{fg}$ meet in a codimension two smooth subscheme (hence a smooth complete intersection). 

\item \label{LA1} $\hbox{codim } (f_i, f_j,g) = 3$ whenever $i \neq j$.

\item \label{LA2} $\hbox{codim } (f, g_i, g_j) = 3$ whenever $i \neq j$.

\end{enumerate}
 (Note that we do not assume that $\hbox{codim } (f_i, f_j, f_k) = 3$ or that $\hbox{codim } (g_i, g_j, g_k) = 3$ for $i, j, k$ distinct.)

Then $\mathcal A$ has the following properties:

\begin{enumerate} [label=(\alph*)]

\item $\displaystyle \Jac(fg)^{top} = \Jac(f)^{top} \cap \Jac(g)^{top} \cap \underbrace{\left [ \bigcap_{i,j} (f_i, g_j) \right ]}_{= \ (f,g)}.$

\item $\displaystyle \Jac(fg)^{top} = g \cdot \Jac(f)^{top} + f \cdot \Jac(g)^{top}$. In particular, $\Jac(fg)^{top}$ is obtained by Liaison Addition from $\Jac(f)^{top}$ and $\Jac(g)^{top}$ (see Theorem \ref{LA}).

\item Statements (a) and (b) continue to hold if we replace $\Jac(fg)^{top}$, $\Jac(f)^{top}$ and $\Jac(g)^{top}$ with their radicals. 

\end{enumerate}
\end{theorem}

Before giving the proof we note that   conditions \ref{LA1} and \ref{LA2} together are equivalent to the following key condition:\\[4pt]
$(\bigstar)\ \left \{ \parbox{5in} {For each associated prime $\mathfrak p$ of $\Jac(f)^{top}$ and each associated prime $\mathfrak p'$ of $\Jac(g)^{top}$, no factor of either $f$ or $g$ is in $\mathfrak p \cap \mathfrak p'$. }
\right.$ \\[4pt]
This is the analog of  the condition in Theorem 3.2 of  \cite{MNS}.

\vspace{.2in}

\begin{proof}[Proof of Theorem \ref{LAthm1}]
By Proposition \ref{compare assoc primes} we know that each associated prime $\mathfrak p$ of $\Jac(f)^{top}$ is also an associated prime of $\Jac(fg)^{top}$, and that the corresponding $\mathfrak p$-primary components of $\Jac(f)^{top}$ and of $\Jac(fg)^{top}$ are equal. Of course the same is true for the associated primes of $\Jac(g)^{top}$. 

Thanks to Equation \eqref{radicals}, it remains only to consider the primary components of $\Jac(fg)^{top}$ supported on any of the smooth codimension two complete intersections $\mathbb V(f_i,g_j)$. 

Fix $i, j$ and consider the polynomials $p = f_i g_j$ and $q = \frac{fg}{p}$. Then
\[
\begin{array}{rcll}
\Jac(fg)^{top} & = & \Jac(pq)^{top} \\
& = & \Jac(p)^{top} \cap \hbox{(other components)} & \hbox{(by Proposition \ref{compare assoc primes})} \\
& = & \Jac(f_i g_j)^{top}  \cap \hbox{(other components)} \\
& = & (f_i, g_j)  \cap \hbox{(other components)} & \hbox{(by Proposition \ref{two factors}}). 
\end{array}
\]
This proves (a). Since $f \in \Jac(f)^{top}$, $g \in \Jac(g)^{top}$ and $(f,g)$ is a regular sequence, part (b) follows from Liaison Addition (Theorem \ref{LA}). Finally, (c) holds because our assumptions give us that $(f,g)$ is already a radical ideal. 
\end{proof}

The following is the second main tool of this paper. 

\begin{corollary} \label{BDL corollary} {\bf (Basic Double Linkage for arrangements.)}
Let $\mathcal A_f$ be a hypersurface arrangement in $\mathbb P^n$, defined  by $f = f_1 \cdots f_s$. Let $g$ be a homogeneous polynomial. Assume:

\begin{enumerate}[label=(\roman*)]

\item Each complete intersection of the form $(f_i, g)$ is smooth.

\item Each complete intersection of the form $(f_i, f_j)$ is smooth.

\item $\hbox{codim } (f_i, f_j,g) = 3$ whenever $i \neq j$.

\end{enumerate}

\noindent Then the following hold:

\begin{enumerate} [label=(\alph*)]

\item $\displaystyle \Jac(fg)^{top} = \Jac(f)^{top} \cap  \underbrace{\left [ \bigcap_{i,j} (f_i, g) \right ]}_{= \ (f,g)}.$

\item $\displaystyle \Jac(fg)^{top} = g \cdot \Jac(f)^{top} + (f)$. In particular, $\Jac(fg)^{top}$ is obtained from $\Jac(f)^{top}$ by Basic Double Linkage (see Theorem \ref{BDL}). 

\item Statements (a) and (b) continue to hold if we replace $\Jac(fg)^{top}$ and $\Jac(f)^{top}$ by their radicals. 

\end{enumerate}

\end{corollary}

\begin{proof}
This is just the case $t=1$ of Theorem \ref{LAthm1}.
\end{proof}


\section{The ACM property for many hypersurface arrangements} \label{sec: main theorem}

The following statement was one of the main results of \cite{MNS}. 

\begin{theorem}[{\cite[Theorem, page 142]{MNS}}] \label{MNS first main thm}
Let $\mathcal A$ be a hyperplane arrangement in $\mathbb P^n$. 
Let $J$ be the Jacobian ideal associated to $\mathcal A$.
Assume that no hyperplane of $\mathcal A$ is in the associated prime of more than one non-reduced  component of $J^{top}$. Then  both  $S/{J^{top}}$  and $ S/\sqrt{J}$ are Cohen-Macaulay. 
\end{theorem}

Our goal now is to see to what extent we can extend this result to hypersurface arrangements. We have seen in \Cref{sat is CI}  that \Cref{plane ci} extends but not in an obvious way. 
Indeed, when we pass from hyperplanes to hypersurfaces, a lot of new complications arise. Nevertheless, using  the results of Section \ref{BDL section},  we are able to generalize Theorem  \ref{MNS first main thm}. This was one of the primary goals of this paper.  The result establishes \Cref{thm:intro-generalization}, stated in the introduction.

\begin{theorem} \label{cond star implies ACM}
Let $\mathcal A$ be a hypersurface arrangement in $\mathbb P^n$, $n \geq 3$, defined by a form $f = f_1 \cdots f_s$.  Let $J = Jac(f)$ be the Jacobian ideal associated to $\mathcal A$. Assume:

\begin{enumerate} [label=(\roman*)]

\item \label{cond 5.2(i)} Any two factors of $f$ meet in a smooth codimension 2 complete intersection.

\item 
No hypersurface of $\mathcal A$ is in the associated prime of more than one non-reduced component of $J^{top}$.

\end{enumerate}

\noindent Then 

\begin{itemize}
\item[(a)]  whenever more than two factors of $f$ meet in a codimension 2 subscheme $Z$, then all such factors have the same degree;

\item[(b)] the schemes defined by  both $J^{top}$ and $\sqrt{J^{top}}$ are ACM;

\item[(c)] If each factor is smooth then also the scheme defined by $\sqrt{J}$ is ACM.

\item[(d)] both ideals $J^{top}$ and $\sqrt{J^{top}}$ define equidimensional schemes that are the union of complete intersections (hence are generically complete intersections). If each $f_i$ is smooth then the same is true also for $\sqrt{J}$.

\end{itemize}

\end{theorem}

\begin{proof}

We first prove (a). Let $g_1,g_2,g_3$ be irreducible factors of $f$ meeting in a codimension 2 subscheme, and without loss of generality assume $\deg(g_1) \leq \deg(g_2) \leq \deg(g_3)$. For convenience set $d_i = \deg(g_i)$ for $i = 1,2,3$. 

If they do not all have the same degree then $d_1 < d_3$. Let $X_1$ be the complete intersection of $g_1$ and $g_2$, which by (ii) is smooth of degree $d_1d_2$. By assumption, $g_3$ also vanishes on $X_1$. Now $g_2$ and $g_3$ also define a smooth (hence irreducible, by ACM) complete intersection, $X_2$, of degree $d_2 d_3 > d_1 d_2 = \deg(X_1)$. This means that $X_2$ provides a non-trivial link of $X_1$ to another codimension 2 scheme, which is impossible since $X_2$ is smooth by (ii). (It is possible that the residual scheme is also supported on $X_1$, but this does not change the fact that  $X_2$ cannot be smooth.)  This proves (a).

Notice that this implies that if more than two irreducible factors of $f$ vanish along a codimension 2 subvariety then that subvariety is the complete intersection of any two of them. 

With the tools we have now developed, the argument follows exactly as it did in Theorem 3.2 and Corollary 3.5 of \cite{MNS}.   For the convenience of the reader we provide the argument with the needed adjustments.   

Any component of $J^{top}$ is smooth (hence reduced) if and only if there are only two factors that vanish on it.  We will first exhaust the non-reduced components, and then  deal with the remaining reduced components.
By Theorem \ref{sat is CI} and by our hypotheses, any of these non-reduced components is a complete intersection. 

Start with a non-reduced component $X_1$ (if it exists). If it is the only non-reduced component then we will go to the next step, and so assume that there is a second non-reduced component, $X_2$. Let $\mathcal A_1$ and $\mathcal A_2$ be the subarrangements defining $X_1$ and $X_2$ respectively (i.e. $\mathcal A_1$ is the union of the hypersurfaces in the radical of $I_{X_1}$ and similarly for $X_2$). Note that both $X_1$ and $X_2$ are ACM, being complete intersections. By (a), the  hypersurfaces in $\mathcal A_1$ all have the same degree $d_1$, and similarly the hypersurfaces in $\mathcal A_2$ all have the same degree $d_2$ (but $d_1$ is not necessarily equal to $d_2$). By assumption  (ii), we are in the  situation  of Theorem \ref{LAthm1}. Letting $J'$ be the Jacobian ideal of $\mathcal A_1 \cup \mathcal A_2$, we thus get that the scheme defined by  $(J')^{top}$  is obtained by Liaison Addition from $X_1$ and $X_2$, and so  $(J')^{top}$ is ACM.
We continue in this way until the  non-reduced components are exhausted. (Along the way  we have also picked up reduced components, but this does not matter.)

For the second step, if any factors of $f$ remain, we add them one at a time, building up our scheme now by Basic Double Linkage (Corollary  \ref{BDL corollary}), and adding only  reduced components. We conclude that  for the Jacobian  ideal $J$ of $\mathcal A$ we get $S/J^{top}$ is Cohen-Macaulay.

For the radical  $\sqrt{J^{top}}$  we note that all of the applications of Liaison Addition and Basic Double Linkage continue to be valid for the radicals of the top dimensional components of the Jacobian ideals, since if $F \in I$ for some homogeneous ideal $I$ then also $F \in \sqrt{I}$. Furthermore, the radical of any of the $X_i$ is also a complete intersection, hence ACM. Thus the same steps as we did for $S/J^{top}$ work for the radical  $S/\sqrt{J^{top}}$. Finally, suppose that each factor is smooth. Then the singular locus of $f$ is supported on the intersections of pairs of factors, so in fact $\sqrt{J} = \sqrt{J^{top}}$ and thus we are done. (Note that $J$ could have embedded components in this support, but they are wiped away by the radical.)  
\end{proof}

\begin{remark}
    Notice two things.  First, by Lemma \ref{consequence of smooth}, Condition \ref{cond 5.2(i)} in Theorem \ref{cond star implies ACM} implies that each factor has at worst 0-dimensional singularities.

    Second, this same Condition \ref{cond 5.2(i)} is automatically true for hyperplane arrangements. Thus Theorem \ref{cond star implies ACM} immediately gives Theorem \ref{MNS first main thm}.
\end{remark}

\begin{example} \label{four cubics}
In Theorem \ref{cond star implies ACM} we assumed $n \geq 3$, so here we look briefly at the case $n=2$. First, the trick used to prove (a) by invoking the smoothness of the complete intersections will not work here because any union of reduced points is smooth, and there is no chance to use irreducibility.

Second, the conclusion in (b) is automatic for $\mathbb P^2$.

Third, recall the elementary result mentioned just after Remark \ref{common factor} above, from \cite[Remark 3.1]{MNS}: if $e$ hyperplanes all contain the same codimension two linear variety $\Lambda$, and if $F$ is the product of the corresponding linear forms, then the Jacobian ideal of $F$ is a saturated complete intersection of type $(e-1,e-1)$. For example, if $F$ is the product of four linear forms then the corresponding complete intersection has degree 9.  Note that we have a generalization of this result to hypersurfaces, namely Theorem \ref{sat is CI}, that applies even in $\mathbb P^2$.

However, one could also hope that the same would be true locally for the non-linear case, at least in $\mathbb P^2$. For example, if $P$ is a point and if $f_1,f_2,f_3,f_4$ are sufficiently general cubics in $I_P$ then locally at $P$ one might hope that the curves associated to these forms could be represented by their tangent lines at $P$, so the scheme defined by the Jacobian is locally a complete intersection of degree 9 at $P$. 

Unfortunately this is not true. We have checked the following on \cocoa. Assume that the four general forms have degree 3, and let $X$ be the scheme defined by the Jacobian ideal of the product of these cubics.  Let $X_1$ be the component of $X$ supported at $P$. Then $X_1$ does indeed have degree 9, but it is not a complete intersection. Indeed, it has four minimal generators, of degrees 3,4,4,4.  So the idea of looking locally and using the tangent lines to determine $X_1$ does not work. Note that this example does not fall into the setting of Theorem \ref{sat is CI}. 

\end{example}

As a consequence  of Theorem \ref{cond star implies ACM}, we obtain a proper generalization of \Cref{reduced case}  from hyperplane arrangements to hypersurface arrangements. 

\begin{corollary} \label{hypersurf config}
Let $\mathcal A$ be a hypersurface arrangement in $\mathbb P^n$ defined by a form $f = f_1 \cdots f_s$. Assume that 
any two  of the hypersurfaces  meet in a smooth complete intersection, and that any three factors meet in codimension 3. Then   the scheme defined by  $\Jac(f)^{top}$ is the reduced union of the complete intersections defined by $(f_i,f_j)$. In particular, it is the ideal of a hypersurface configuration (in the language of \cite{GHMN}), so it is ACM and its minimal generators are the forms $\frac{f}{f_i}$.
\end{corollary}
 
\begin{proof}
 The assumption that any two of the hypersurfaces meet in a smooth complete intersection means that non-reduced components of $J^{top}$ (which all have codimension two) must involve at least three of the hypersurfaces. The assumption here that  any three of the hypersurfaces meet in codimension three excludes this. Thus the components of $J^{top}$ are all reduced.  In particular, condition (ii) of Theorem \ref{cond star implies ACM} is vacuous.   The first assertion is a result of Theorem \ref{cond star implies ACM} (see also Corollary \ref{BDL corollary}). The second assertion follows from \cite[Proposition 2.3]{GHMN}.
\end{proof}

It is natural to ask whether Theorem \ref{cond star implies ACM} also holds if we replace $\Jac(f)^{top}$ by $\Jac(f)^{sat}$. The answer is no: the saturation may have embedded points. This is true even for hyperplane arrangements, as the following example shows.

\begin{example}
Let  $P$ be the point in $\mathbb P^3$ defined by the ideal $(x,y,z)$. Let $L_1,\dots,L_4$ be four general linear forms in $I_P$ and let $F$  be their product. Let $J$ be the Jacobian ideal of $F$. 

Of course the scheme defined by $J$  has degree 6, and its top dimensional part is supported on the six lines defined by pairs of the linear forms. But the saturation has Hilbert function  $6t-1$  (and in fact $J$ is already  saturated) while the top dimensional part has Hilbert function  $6t-2$ (and is clearly  ACM, being a cone over a set of 6 points in $\mathbb P^2$). So there is an embedded component supported at $P$.
\end{example}

\section{Illustrative examples related to Theorems \ref{sat is CI}  and \ref{cond star implies ACM} }  \label{sect: examples}


In this section we illustrate how the conclusions in our main statements may fail if one relaxes some of the assumptions. 

One of the main results of this paper is Theorem \ref{sat is CI}, which is instrumental in many of the subsequent results. Moreover, notice that our main result, Theorem \ref{cond star implies ACM}, formally gives the conclusion  of Theorem \ref{sat is CI} about the  Cohen-Macaulayness of $\Jac(f)^{top}$ as a special case. Thus any generalization of Theorem \ref{cond star implies ACM} appears to require a generalization of Theorem \ref{sat is CI}. However, we will see in this  section that the conclusion of Theorem \ref{sat is CI}  (and thus of \Cref{cond star implies ACM})  may fail in surprising ways if one relaxes some of its assumptions. See also Example \ref{four cubics}. 

The essential hypotheses of Theorem \ref{sat is CI} are that  $Z$ is a smooth complete intersection and that the minimal generators of $I_Z$ have the same degree. 
More precisely, the statement of Theorem \ref{sat is CI} is that taking $G_1,\dots,G_s \in [I_Z]_d$  all smooth, and setting $G = \prod_{i=1}^s G_i$, then the saturation of the Jacobian ideal of $G$ is a complete intersection (and we describe how far it is from already  being saturated). In particular, this saturation is unmixed and ACM.
In this section we give some examples to show that the result is not true without these assumptions.

All of the examples in this  section were computed using the computer algebra package \cocoa\ \cite{cocoa}, and if no theoretic argument is given, the examples are just reports of the results of the calculations.

\begin{example} \label{counterexa}
    This example shows that if we allow the degrees of the hypersurfaces defining the complete intersection to be different,  the result of Theorem \ref{sat is CI} need not hold. 
    
    Let $F$ be a general form of degree 2, $G$ a general form of degree 3, and $H_1,H_2,H_3 \in (F,G)$ general elements of degree 4. Note that $(F,G)$ defines a smooth complete intersection, $X$, of degree 6, and all five hypersurfaces are smooth. However,  just the degrees alone force the other pairs of hypersurfaces to meet in  singular complete intersections. (For example, $(F,H_1)$ links $X$ to a residual curve $Y$ of degree 2 so $(F,H_1)$ defines $X \cup Y$, which is singular.) Let $J$ be the Jacobian ideal of the polynomial $FGH_1H_2H_3$.

The following conclusions for the Jacobian ideal $J$ and the top dimensional part $J^{top}$ hold.

\begin{itemize}
    \item $J$ is not saturated.

    \item The saturation of $J$ is not unmixed: the corresponding curve has embedded points. Indeed, the Hilbert polynomial of $R/J$ is $138t -759$ while the Hilbert polynomial of $R/J^{top}$ is $138 t - 1217$.

    \item $J^{top}$ defines a non-ACM curve of degree 138.

    \item The component supported on $X$ of the scheme defined by $J^{top}$ has degree 84 and is  not ACM.
\end{itemize}

\noindent In sum, almost nothing from Theorem \ref{sat is CI} can be expected to remain true without the assumption on the  degrees.

\end{example}

\begin{example}

The  purpose of this  example is to show that even the  degree of the scheme defined by  $J^{top}$ can fail to be what one might expect.

Let $\ell$ be a line in $\mathbb P^3$. Let $F \in [I_\ell]_3$ be a cubic surface that is smooth at all points of $\ell$. Let $D$ be the non-ACM double line with saturated ideal $I_D = I_\ell^2 + (F)$ (see for instance \cite{dble}). Let $G_1 \in [I_D]_4$ and $G_2 \in [I_D]_5$ be general choices; in particular, they form a regular sequence and they are tangent along $\ell$ (since $D$ contains information about the tangent direction).  It turns out that $G_1$ has one singular point and $G_2$ has two (distinct) singular points.

Notice that the complete intersection $(G_1,G_2)$ links $D$ to a residual (non-ACM, by liaison) curve $E$ of degree 18; $E$ turns out to be smooth. Denoting by $A$ this complete intersection, note that $A = E \cup D$ and 
\[
20 = \deg(A) = 18 + 2 = \deg (E) + \deg(D).
\]

Now let $\mathcal A$ be the hypersurface arrangement defined by $G = G_1 \cdot G_2$, with $G_1, G_2$ as just defined. As noted, they meet in a reducible complete intersection curve of degree 20 that is non-reduced along a line. Let $J$ be the Jacobian ideal of $G$. It turns out  that $J^{top}$ defines a curve $X^{top}$ of degree 21 (not 20). It consists of the same $E$ together with a curve $Y$ of degree 3 (not 2) supported on $\ell$. In fact, $Y$ contains $D$ as a subscheme, and we have
\[
X^{top} = E \cup Y.
\]
So the process of taking the Jacobian ideal and then the top dimensional part does not simply produce  the complete intersection of $G_1$ and $G_2$, and the obstacle arises only along the tangent locus of the two hypersurfaces. 

And finally, in fact $X^{top}$ is not ACM! Nor can it be produced using basic double linkage. 
Thus we need a condition that was not necessary in the case of hyperplanes. For this reason we assume  in \Cref{cond star implies ACM}  that any two of the hypersurfaces in $\mathcal A$ meet along a smooth complete intersection.

\end{example}

\begin{example} \label{generalize "sat is CI" 1}

At the heart of the proof of \Cref{cond star implies ACM} is the surprising result of \Cref{sat is CI}. 
In this  example we  illustrate what may happen if we  slightly tweak the assumption of Theorem \ref{sat is CI}. Namely, we start with a complete intersection (a line) in $\mathbb P^3$ but instead of taking a union of surfaces of minimal degree containing it, we increase the degree by one. Specifically, let $\lambda$ be a line in $\mathbb P^3$ and let $G_1, G_2, G_3$ be three general quadrics containing it. Let $G = G_1 G_2 G_3$ and let $J$ be the Jacobian ideal of $G$. Since two quadrics link a line to a twisted cubic, we expect $J$ to define (up to saturation and up to possible embedded points) a non-reduced scheme $X$ supported on $\lambda$ together with the union of three twisted cubics, $C_1, C_2, C_3$. 

Our computations on \cocoa\ give specifically that $J$ is not saturated, but $J^{sat}$ is unmixed and defines curves supported on  $\lambda, C_1, C_2, C_3$ as predicted. We further see that $R/J^{sat}$ is not ACM, and that neither $X$ nor $C_1 \cup C_2 \cup C_2$ is ACM. In particular, $X$ has degree 4 but is not a complete intersection, as might have been predicted from Theorem \ref{sat is CI}.

\end{example}

The following example suggests that trying to move beyond the  results obtained in this  paper will not be so easy. It is not clear what the pattern is, at least in the direction we explore below.

\begin{example}  \label{ex: twc}
Another possible idea to try to generalize \Cref{cond star implies ACM} might be to replace the complete intersections by something other than a complete intersection. The simplest example might be allowing one of the base loci to be a twisted cubic curve. Here we explore this and see what may be the ramifications of making this choice. In contrast to Example \ref{generalize "sat is CI" 1}, we again take surfaces of smallest possible degree so the only difference is that the base locus is not a complete intersection. 
Thus we illustrate the importance of the complete intersection assumption in Theorem \ref{sat is CI}. 

Let $C$ be a twisted cubic curve in $\mathbb P^3$ and let $G_1,\dots,G_6 \in [I_C]_2$ be general quadrics containing $C$.  First let $G = G_1\cdots G_4$ and let $J$ be the Jacobian ideal of $G$. Then $J$ is already a saturated ideal. Its Hilbert polynomial is $30t - 90$. However, the top dimensional part $J^{top}$ of $J$ has Hilbert polynomial $30t - 96$, meaning that $J$ has embedded components. One checks that the scheme defined by $J^{top}$ is ACM. It consists of the union of six lines and a non-reduced component of degree 24 supported on $C$. This non-reduced component by itself is not ACM. 

The existence of the six lines can be explained since there are $\binom{4}{2}$ pairs of quadrics among the $G_i$, and each pair gives a link of $C$ to a line. The non-ACMness of the component supported on $C$ offers a contrast with the result of Theorem \ref{sat is CI}.

One could hope that at least the ACMness of the top dimensional part is in analogy with Theorem \ref{sat is CI}. So  let $G' = G_1\cdots G_6$ and let $J'$ be the Jacobian ideal of $G'$. Now $J'$ is not saturated. Its saturation $(J')^{sat}$ has Hilbert polynomial $78t -494$. One checks that $(J')^{sat}$ is unmixed, but the scheme it defines is not ACM. It consists of the union of 15 lines and a non-reduced curve of degree 63 supported on $C$. This non-reduced curve is again not ACM.

It is interesting to see how different the results are between taking four elements of $[I_C]_2$ and taking six such elements. Indeed, the absence of a clear pattern is striking. More precisely, we continue to let $C$ be a twisted cubic curve, and we take $G_1,\dots,G_s \in [I_C]_2$ to be general choices, for different values of $s$. We set $\mathcal A$ to be the hypersurface arrangement defined by $G_1\cdots G_s$.

Note that the residual to a twisted cubic, $C$, in a general complete intersection of quadrics in $[I_C]_2$ is a line, so as before $X^{top}$ and $X^{red}$  have $\binom{s}{2}$  components that are lines.
Of course when $s=2$ the singular locus of $G_1 G_2$ is a reduced complete intersection linking the twisted cubic to a line, and $X^{red} = X^{top}$ is ACM. So we are more interested in $s \geq 3$. Then $X^{top}$ will consist of a non-reduced component, say $Y$, and a union of $\binom{s}{2}$ lines.

Using \cocoa\ in the range $3 \leq s \leq 11$ we find that:

\begin{itemize}

\item Sometimes the Jacobian is saturated ($s = 4,5$) and otherwise it is not. \smallskip

\item Sometimes the saturation is unmixed ($s = 3,5,6,8,10$) and otherwise it is not. \smallskip

\item Sometimes the entire top dimensional part $X^{top}$ (the union of $Y$ and the lines) is ACM ($s = 3,4$) and otherwise it is not.  It seems that for $s \geq 5$ it always fails to be ACM. \smallskip

\item Even $\deg(Y)$,  the degree of the non-reduced component supported on the twisted cubic, and $\deg(X^{top})$ (the union of $Y$ and the skew lines) follow a strange pattern: 

\vspace{.1in}

\begin{center}

\begin{tabular} {c|ccccccccccccccccccc}
$s$ & 3 & 4 & 5 & 6 & 7 & 8 & 9 & 10 & 11  \\ \hline
$\deg(Y)$ & 12 & 24 & 42 & 63 & 87 & 117 & 150 & 189 & 231 \\
$\deg(X^{top})$ & 15 &30 & 52 & 78 & 108 & 145 & 186 & 234 & 286
\end{tabular}

\end{center}

\vspace{.1in}

\item The only consistent fact in this range is that the non-reduced component, $Y$, supported on the twisted cubic is not ACM for all $3 \leq s \leq 11$. However,  for $s=3,4$ when we add the $s$ lines, the result ($X^{top}$) becomes ACM.

\end{itemize}

Again, the conclusion here is that it will be very difficult to obtain further positive ACM results for the top dimensional part of a hypersurface arrangement, once one allows three of them to vanish on the same codimension two component that is not the complete intersection of two of them. 
Of course we would need to choose as our building blocks not just the non-reduced component supported on the twisted cubic but, instead, the reducible curve $X^{top}$ consisting of the union of this non-reduced component and $\binom{s}{2}$  lines. But as noted, even this is not  ACM for $s \geq 5$ (in our data range). Thus we feel that it will be difficult to find a more general set of hypotheses that force $X^{top}$ and $X^{red}$ to be ACM. It may be possible to show that in general (perhaps for $s$ sufficiently large), $X^{top}$ is {\em not} ACM, and perhaps that $X^{red}$ {\em is} ACM (depending on whether the union of $C$ and the $\binom{s}{2}$ lines is ACM). So there may be room for further progress here.

\end{example}

Our final example again illustrates why we have the assumptions on reducedness. Of course one immediately sees that the conclusion $J^{top} = \sqrt{J}$ is impossible if $J^{top}$ is not reduced, but there is a more subtle issue that we want to illustrate, involving the ACM conclusion.

\begin{example} \label{kummer}
It is well known (see the introduction of \cite{rao self-link}) that there is a smooth, non-ACM curve, $Z$, of degree 8 and genus 5 that is linked to itself in the complete intersection of two Kummer surfaces (which have degree 4).  Let $\mathcal A$  be the union of these two surfaces.  Let $J$ be its Jacobian ideal.  Theorem \ref{sat is CI}, or the other results of Section \ref{one CI}, do not quite apply since the complete intersection of the two surfaces is not reduced, and even the support is not a complete intersection  (or ACM, although it is smooth).  Surprisingly, it still turns out  that $J^{top}$ is  a complete intersection of type $((s-1)d,(s-1)d) = (4,4)$ so $J^{top}$ is ACM. But by construction, $\sqrt{J}$ is not ACM. The surfaces used in this example are not smooth, but nevertheless they have only nodes as singularities (so they do not affect $J^{top}$) and the fact that they are tangent along a curve illustrates one of the difficulties that we face in passing from hyperplanes to hypersurfaces. 

\end{example}

\section{Arrangements in $\mathbb P^3$ and Hartshorne-Rao modules} \label{sec: HR}

In this  section we restrict to the polynomial ring in four variables, $S = K[x_0,x_1,x_2,x_3]$, and consider hypersurface (including hyperplane) arrangements in $\mathbb P^3$ for simplicity. We note though that the results in this section have analogs for arrangements in any projective space $\PP^n$ with $n \ge 3$.  If $J$ is the Jacobian ideal of  a hypersurface arrangement in $\mathbb P^3$ then both ideals $J^{top}$ and $\sqrt{J}$ are unmixed, hence define  equidimensional curves in $\mathbb P^3$. Such a  curve $C$  has an even liaison class, which (up to shift) is identified with a finite length graded module $M(C)$. Specifically,  we have the  following definition.

\begin{definition}
Let $C \subset \mathbb P^3$ be an equidimensional curve with no embedded points. Then we denote by $M(C)$ the {\em Hartshorne-Rao module} of $C$, namely the graded module 
\[
M(C) = \bigoplus_{t \in \mathbb Z} H^1(\mathbb P^3, \mathcal I_C(t)). 
\]

\end{definition}

Under our unmixedness assumptions on $C$, $M(C)$ is a graded module of finite length. It is zero if and only if $C$ is ACM. The even liaison classes of equidimensional curves in $\mathbb P^3$ are in bijective correspondence with the finite length graded $R$-modules up  to shift. More specifically, curves $C$ and $C'$ are in the same even liaison class if and only if $M(C) \cong M(C')(\delta)$ for some $\delta \in \mathbb Z$.
See \cite{rao} for details. 

It is an interesting question to ask which even liaison  classes contain curves arising from Jacobian ideals, either as $J^{top}$ or as $\sqrt{J}$. Thus we focus in this section on the possible Hartshorne-Rao modules of such curves. This question  was addressed in \cite{MNS} in the case of hyperplane arrangements.

In this section we will use the notation $\mathcal A_g$ for the arrangement in $\mathbb P^3$ defined by a product $g$ of homogeneous polynomials, and by $C_g^{top}$ and $C_g^{red}$ the curves defined by the ideals $J^{top}$ and $\sqrt{J}$.


With this notation, we  obtain two results about the behavior of Hartshorne-Rao modules associated to modified hypersurface arrangements.  In particular, the first result provides  a method for producing arrangements with large Hartshorne-Rao modules. 

\begin{proposition} \label{prop:HR liaison addition}
    Let $f = f_1 \cdots f_s$ and $g = g_1 \cdots g_t$ be products of homogeneous polynomials such that

\begin{enumerate}[label=(\roman*)]

\item Any two distinct hypersurfaces in $\mathcal A_{fg}$ meet in a codimension two smooth subscheme (hence a smooth complete intersection). 

\item  $\hbox{codim } (f_i, f_j,g) = 3$ whenever $i \neq j$.

\item  $\hbox{codim } (f, g_i, g_j) = 3$ whenever $i \neq j$.

\end{enumerate}

Let $d_1 = \deg f$ and $d_2 = \deg g$. Then 
\[
M(C_{fg}^{top}) = M(C_f^{top}) (-d_2) \oplus M(C_g^{top})(-d_1)
\]
and
\[
M(C_{fg}^{red}) = M(C_f^{red}) (-d_2) \oplus M(C_g^{red})(-d_1)
\]
In particular, if $C^{top}_f$ and $C^{top}_g$ are both ACM then so is $C_{fg}^{top}$. The analogous statement holds for $C^{red}_{fg}$.
\end{proposition}

\begin{proof} 
Theorem \ref{LAthm1} guarantees that $C_{fg}^{top}$ is indeed obtained by a Liaison Addition, and then the result follows from Theorem \ref{LA}.
\end{proof}

As a consequence, we see that adding a general  hypersurface to an arrangement preserves the even liaison class. In particular, it preserves the Cohen-Macaulay property if present originally.

\begin{corollary} \label{HR mod shift}
Let $\mathcal A_f$ be a hypersurface arrangement in $\mathbb P^3$, defined  by $f = f_1 \cdots f_s$. Let $g$ be a homogeneous polynomial of degree $d$. Assume:

\begin{enumerate}[label=(\roman*)]

\item Each complete intersection of the form $(f_i, g)$ is smooth.

\item Each complete intersection of the form $(f_i, f_j)$ is smooth.

\item $\hbox{codim } (f_i, f_j,g) = 3$ whenever $i \neq j$.

\end{enumerate}

\noindent Then 
\[
M(C_{fg}^{top}) = M(C_f^{top}) (-d) 
\]
and
\[
M(C_{fg}^{red}) = M(C_f^{red}) (-d) .
\]
In particular, if $C_{f}^{top}$ is ACM then so is $C_{fg}^{top}$, and the analogous  statement holds for $C_{fg}^{red}$. More generally, the even liaison class of $C_f^{top}$ is the same as the even liaison class of $C_{fg}^{top}$, 
 and the even liaison class of $C_f^{red}$ is the same as that of $C_{fg}^{red}$. 

\end{corollary}

\begin{proof}
Corollary \ref{BDL corollary} guarantees that a basic double link is being performed, and so the assertions all follow from Proposition \ref{BDL}.
 We only note that the components coming from factors of $f$ are already accounted for in $ C_f^{top}$ (resp. $C_f^{red}$), and because of the assumption on $\mathcal A_f$ the new components are smooth curves coming from the intersection of 
 $f$ and $g$. 
\end{proof}

In \cite{MNS},  the authors gave an example of a hyperplane arrangement in $\mathbb P^3$ for which  $J^{top}$ defines a scheme whose Hartshorne-Rao module is one-dimensional, and one for which $\sqrt{J}$ defines a scheme whose Hartshorne-Rao module is one-dimensional. Then using Liaison Addition, arrangements were produced whose corresponding curves fell into any Buchsbaum liaison class. It is natural to wonder if this  can be done with hypersurfaces of larger degree, and in the next example we explore this  question.

\begin{example} \label{gen MS}
The following is motivated by  \cite[Example 4.5]{MS} (see also \cite[Example 4.1]{MNS}). There, a hyperplane arrangement is given for which  $J = J^{top}$ is not ACM, and has a one-dimensional Hartshorne-Rao module. (In that example, $\sqrt{J}$ was ACM, but related examples were found in \cite{MNS} for which $\sqrt{J}$ is not ACM.) Here we take a subset of that example, motivated by \cite[Example 4.3]{MNS} but replacing the linear forms by general quadratic forms. We will see that both $J^{top}$ and $\sqrt{J}$ fail to beACM.

To stress the connection with those examples, we denote by $X,Y,Z,W$ the general quadrics. We let
\[
G := X*Y*Z*W*(X+Y)*(Y+Z)*(Z+W)*(W+X)*(W+X+Y+Z)
\]
and take the Jacobian ideal, $J$, generated by the four partial derivatives of $G$. This gives an ideal whose quotient has Hilbert polynomial $168t - 1728$. It turns out that the saturation, $J^{sat}$, has the same Hilbert polynomial but not the same Hilbert function, so $J$ is not saturated. It turns out that the saturation is unmixed: $J^{sat} = J^{top}$. The Betti diagram of $R/J^{top}$ is

\begin{verbatim}
         0    1    2    3
-------------------------
  0:    1    -    -    -
  1:    -    -    -    -
   ...
 14:    -    -    -    -
 15:    -    4    -    -
 16:    -    -    -    -
   ...
 19:    -    -    -    -
 20:    -    -    4    -
 21:    -    -    -    1
-------------------------
Tot:    1    4    4    1

\end{verbatim}

Notice that  the last matrix in this resolution is a $1 \times 4$ matrix of quadrics, and one can check that it defines a complete intersection. So the Hartshorne-Rao module of $C_G^{top}$ is a complete intersection. Its components have dimensions $(1,4,6,4,1)$.

This example suggests (for example) that there might not exist a quadric arrangement yielding a Hartshorne-Rao module that is one dimensional. Indeed, the Musta\c t\v a-Schenck example gave a Hartshorne-Rao module that was isomorphic to $S/(x,y,z,w)$, and replacing the variables by quadrics gives a Hartshorne-Rao module  isomorphic to $S/(Q_1,Q_2,Q_3,Q_4)$, where the $Q_i$ are quadrics. 

The radical of $J$ also turns out to not be ACM. Its Hilbert polynomial is $96t -672 $ and its Betti diagram is

\begin{verbatim}
        0    1    2    3
-------------------------
  0:    1    -    -    -
  1:    -    -    -    -
  ...
 12:    -    -    -    -
 13:    -   12    -    -
 14:    -    -   15    -
 15:    -    -    -    4
-------------------------
Tot:    1   12   15    4
\end{verbatim}

\end{example}

Using further quadruples of general quadrics to produce  hypersurface arrangements analogous to \Cref{gen MS} and then employing  \Cref{prop:HR liaison addition}, we obtain a hypersurface arrangement whose Hartshorne-Rao modules has a dimension exceeding any finite bound.  This also motivates the following question. 

\begin{problem}
    Which even liaison classes contain the top-dimensional part of a hypersurface arrangement? 
\end{problem}

Bearing in mind that Liaison Addition and Basic Double Linkage also involve shifts of the modules, we pose the following question that partially invokes the Lazarsfeld-Rao property for even liaison classes \cite{BBM}.

\begin{problem} \label{liaison problem}
    Let $\mathcal A$ be a hypersurface arrangement in $\mathbb P^3$ defined by a product $f$ of homogeneous polynomials satisfying the hypotheses of Theorem \ref{cond star implies ACM}. Let $M = M(C_f^{top})$. Is it true that {\em up to flat deformation}, $C_f^{top}$ is the only curve arising as the top-dimensional part of a hypersurface arrangement and having Hartshorne-Rao module $M$? (An analogous question can be asked in $\mathbb P^n$.)
\end{problem}

\begin{remark}
\begin{enumerate}
    
\item Note that Liaison Addition (Theorem \ref{LAthm1}) and Basic Double Linkage (Corollary \ref{BDL corollary}) can produce many other curves arising as the top dimensional part of an arrangement, having Hartshorne-Rao module that is a {\em shift} of $M$. This question asks about $M$ itself. 

Furthermore, if $\mathcal A_f$ is such an arrangement and $\ell_1, \ell_2$ are both general linear forms, then it follows from what we have shown in this paper that $C_{\ell_1f}^{top}$ and $C_{\ell_2 f}^{top}$ are different curves having isomorphic Hartshorne-Rao modules, namely isomorphic to $M(-1)$. However, these curves lie in a flat family, so they do not give a counterexample to our question.

\item In \cite{MN11} the authors showed that within any even liaison class of codimension two subschemes of $\mathbb P^n$, the subset of those elements that satisfy a certain numerical property actually has a Lazarsfeld-Rao property of its own. In view of what we have just said, we can further ask the following. 

It follows immediately from Corollary \ref{HR mod shift} that once you have a subscheme in an even liaison class coming from an arrangement, then that even liaison class has infinitely  many other subschemes, of infinitely many  degrees, coming from arrangements. The point of our Problem \ref{liaison problem} is to give some structure to the set of all subschemes in an even liaison class coming from  arrangements, and in  particular to ask whether there are any that do not come directly  from Corollary  \ref{HR mod shift}.

More precisely, let $C_f^{top}$ be a subscheme coming from a hypersurface arrangement defined by $f$, and let $\mathcal L$ be the even liaison class of $C_f^{top}$. Let $\mathcal M \subset \mathcal L$ be the set of elements of $\mathcal L$ that arise as  the top dimensional parts of hypersurface arrangements. Does $\mathcal M$ also satisfy a Lazarsfeld-Rao property? See \cite{BBM} and  \cite{MN11} for details. 

\end{enumerate}
\end{remark}



\end{document}